\documentclass[12pt,a4paper,english]{amsart}
\usepackage{amsmath} 
\usepackage{amsthm} 
\usepackage{amssymb} 
\usepackage{amscd} 
\usepackage{array} 
\usepackage[mathscr]{eucal} 
\usepackage{hyperref} 
\usepackage{caption} %
\usepackage{graphics,graphicx} 
\usepackage{tikz} 
\usepackage{enumerate} 
\usepackage{tikz} 
\usepackage{float}
\usepackage[margin=2.4cm]{geometry}

\numberwithin{equation}{section}
\usetikzlibrary{automata}

\theoremstyle{plain}
\newtheorem{teo}{Theorem}[section]
\newtheorem{lema}[teo]{Lemma}
\newtheorem{coro}[teo]{Corollary}
\newtheorem{prop}[teo]{Proposition}

 \theoremstyle{definition}

\newtheorem{conj}[teo]{Conjecture}
\newtheorem{obs}[teo]{Remark}
\newtheorem{?}[teo]{Problem}

\newcommand{\rk}{\operatorname{rk}}

\allowdisplaybreaks

\begin{document}

\title{On the Ehrhart Polynomial of Minimal Matroids}

\author[L. Ferroni]{Luis Ferroni}
\thanks{The author is supported by the Marie Sk{\l}odowska-Curie PhD fellowship as part of the program INdAM-DP-COFUND-2015, Grant Number 713485.}

\address{Universit\`a di Bologna, Dipartimento di Matematica, Piazza di Porta San Donato, 5, 40126 Bologna BO - Italia} 

\email{ferroniluis@gmail.com\\ luis.ferronirivetti2@unibo.it}


\begin{abstract} 
	We provide a formula for the Ehrhart polynomial of the connected matroid of size $n$ and rank $k$ with the least number of bases, also known as a \textit{minimal matroid}. We prove that their polytopes are Ehrhart positive and $h^*$-real-rooted (and hence unimodal). We prove that the operation of circuit-hyperplane relaxation relates minimal matroids and matroid polytopes subdivisions, and also preserves Ehrhart positivity. We state two conjectures: that indeed all matroids are $h^*$-real-rooted, and that the coefficients of the Ehrhart polynomial of a connected matroid of fixed rank and cardinality are bounded by those of the corresponding minimal matroid and the corresponding uniform matroid.\\
	
\end{abstract}

\maketitle

\section{Introduction} 

In the article \cite{deloera} De Loera, Haws and K\"oppe posed the following conjecture:

\begin{conj}\label{deloera}
	The Ehrhart polynomial of every matroid polytope has positive coefficients and the $h^*$-vector is unimodal.
\end{conj}

The main reason behind this hypotheses, according to the authors, was computational evidence provided by the software {\tt LattE} \cite{latte}.

Further evidence in favor of the first part of this conjecture was established recently in \cite{ferroni}, where the following result is proven:

\begin{teo}
	The Ehrhart polynomial of the matroid polytopes of every uniform matroid has positive coefficients.
\end{teo}

In \cite{fuliu} Castillo and Liu conjectured something stronger regarding Ehrhart positivity: that all \textit{generalized permutohedra} are Ehrhart positive. The validity of that conjecture implies that all matroids are Ehrhart positive, since it is known that matroid polytopes are a subfamily of generalized permutohedra (see for example \cite{ardila}).\\

In this article we are going to support Conjecture \ref{deloera} by proving that a certain infinite family of matroids with nice properties does satisfy that assertion. Moreover, we prove that in this case such matroids have polytopes that are $h^*$-real-rooted. This fact motivated further computations on several matroids. We verified using {\tt LattE} that the following matroids are $h^*$-real-rooted.
	\begin{itemize}
		\item The matroid polytope of every matroid with up to $9$ elements.
		\item The matroid polytope of every uniform matroid (every \textit{hypersimplex}) with up to $200$ elements.
		\item All matroids listed in \cite{deloera}.
		\item All \textit{snake}-matroids \cite{knauer1} of the form $S(a_1,\ldots,a_{\ell})$ for $a_1+\ldots+a_{\ell}\leq 22$.
		\item All Lattice Path Matroids \cite{bonin} with up to $12$ elements. 
	\end{itemize}
Therefore, since in such a case real-rootedness implies log-concavity, and this in turn implies unimodality, we can state a more general conjecture:

\begin{conj}\label{realrootedness}
	The $h^*$-polynomial of every matroid polytope is real-rooted. 
\end{conj}

Since there are several results regarding real-rootedness of polynomials in combinatorics, see for instance Petter Br\"and\'en's articles \cite{branden2} and \cite{branden1} and more recently  \cite{branden3} joint with L. Solus, Jochemko's \cite{jochemko} and \cite{beckjochemko} by Beck, Jochemko and McCullough, it seems that approaching this Conjecture from that point of view may allow to use much more machinery.\\

To begin our discussion, we recall that a matroid $M$ is said to be \textit{connected} if it cannot be decomposed as a direct sum of smaller matroids. It is known (see for instance \cite{feichtner}) that the dimension of the matroid polytope of a matroid $M$ is $|M|-c(M)$ where $|M|$ is the cardinality of the ground set of the matroid and $c(M)$ the number of \textit{connected components} of the matroid (see \cite{oxley} for undefined terminology).

Also, the following result is a straightforward consequence of the definitions.

\begin{prop}
	If $M_1$ and $M_2$ are matroids with matroid polytopes $\mathscr{P}_1$ and $\mathscr{P}_2$ respectively, then the matroid polytope of the direct sum of matroids $M_1\oplus M_2$ is the product of polytopes $\mathscr{P}_1\times \mathscr{P}_2$. In particular, the Ehrhart polynomial of $M_1\oplus M_2$ is the product of the Ehrhart polynomials of $M_1$ and $M_2$.
\end{prop}

Therefore, if one proves that every connected matroid is Ehrhart positive, the first part of Conjecture \ref{deloera} follows immediately.\\

Subsequently, for every pair of integers $n$ and $k$, we denote $\mathscr{C}(k,n)$ the family of all classes of isomorphism of connected matroids of size $n$ and rank $k$.

Notice that for every $M\in \mathscr{C}(k,n)$, the set of bases of $M$, denoted by $\mathscr{B}(M)$, has size at most $\binom{n}{k}$, and equality is attained only when $M$ is isomorphic to the uniform matroid $U_{k,n}$.\\

Hence, the basis polytope $\mathscr{P}(M)$ of a matroid in $\mathscr{C}(k,n)$ is \textit{contained} (as a set) in the hypersimplex $\Delta_{k,n}:=\mathscr{P}(U_{k,n})$. In particular, the number of lattice points in every $t\mathscr{P}(M)$ for $t\geq 0$ is at most the number of lattice points in the dilated hypersimplex $t\Delta_{k,n}$. This gives place to a natural question: is it true that the coefficients of the Ehrhart polynomial of the hypersimplex are always greater or equal than the corresponding coefficients of another matroid of the same rank and cardinality?\\

Turning things around, according to a result established independently by Dinolt \cite{dinolt} and Murty \cite{murty} in the 70's, there is exactly one element in $\mathscr{C}(k,n)$ having the least number of bases. Throughout this article we denote these matroids by $T_{k,n}$ and following \cite{dinolt} we call them \textit{minimal matroids}. \\

If we use the notation $P(t)\preceq Q(t)$ on polynomials $P$ and $Q$ to denote that for every $m$ the coefficient of degree $m$ in $P$ is less or equal than the coefficient of degree $m$ in $Q$, we state our conjecture as follows:
 
\begin{conj}\label{cota}
	Let us denote $i(M,t)$ the Ehrhart polynomial of a matroid $M$. Then if $M$ is a connected matroid of rank $k$ and cardinality $n$, the following inequality holds:
		\[ i(T_{k,n},t) \preceq i(M,t) \preceq i(U_{k,n},t).  \]
\end{conj}

We were able to establish here the Ehrhart positivity (and the $h^*$-real-rootedness) of $T_{k,n}$, and to provide a manifestly positive formula for the coefficients. Hence, the first inequality on Conjecture \ref{cota} implies the first assertion of Conjecture \ref{deloera}.\\

The Ehrhart positivity of minimal matroids is a clue pointing to the truthfulness of De Loera's et al. Conjecture. So far the only known infinite family of Ehrhart positive matroids of all ranks and cardinalities were uniform matroids \cite{ferroni}. This also provides a new example.

The validity of Conjecture \ref{cota} would provide then a severe restriction for the possible polynomials that may occur as the Ehrhart polynomial of a (connected) matroid. 

As we mentioned before, the polytope $\mathscr{P}$ of a connected matroid of rank $k$ and cardinality $n$ is contained in the hypersimplex $\Delta_{k,n}$. In spite of that, it is not true in general that $\mathscr{P}$ contains a copy of the matroid polytope of $T_{k,n}$.

Still, our inequalities have been verified for the all the matroids on the list we gave above.\\

We state our main results as follows:

\begin{teo}
	Let us denote $D_{k,n}(t)$ the Ehrhart polynomial of the matroid $T_{k,n}$. Then:
		\[D_{k,n}(t) = \frac{1}{\binom{n-1}{k-1}} \binom{t+n-k}{n-k} \sum_{j=0}^{k-1}\binom{n-k-1+j}{j}\binom{t+j}{j}\]
	All coefficients of $D_{k,n}(t-1)$ are positive, and hence so are the coefficients of $D_{k,n}(t)$.
\end{teo}

\begin{coro}
	The $h^*$-polynomial of $\mathscr{P}(T_{k,n})$ is given by the formula:
	\[h^*(T_{k,n},x) = \sum_{j=0}^{k-1} \binom{k-1}{j}\binom{n-k-1}{j} x^j. \]
	Moreover, it is real-rooted and hence the $h^*$-vector is log-concave and unimodal.
\end{coro}

In the course of our proofs, we give a graphical realization of $T_{k,n}$ and an inequality characterization of its matroid polytope.\\

Also, we briefly recall (see Section 5) that if a matroid has a hyperplane that is also a circuit, one may declare this subset to be a basis. This operation, called \textit{relaxation}, changes the matroid by adding just one basis. Thus, if we think of the polytope, this operation consists of adding one more vertex. We prove that this construction is essentially \textit{gluing} the polytope of a minimal matroid, and moreover:

\begin{teo}
	If $M\in\mathscr{C}(k,n)$ is an Ehrhart positive matroid and $\widetilde{M}$ is a circuit-hyperplane relaxation of $M$, then $\widetilde{M}$ is Ehrhart positive too. Furthermore, the following equality holds:
		\[ i(\widetilde{M},t) = i(M,t) + D_{k,n}(t-1)\]
\end{teo}

As a consequence of the preceding theorem we present a way of constructing examples of non-dual and non-isomorphic connected matroids that have the same Ehrhart polynomial and also the same Tutte polynomial, and whose basis polytopes are not combinatorially equivalent.

\section{The matroid polytope of minimal matroids}

We start this section by recalling a result established independently in \cite{dinolt} and \cite{murty}.

\begin{teo}\label{dinoltmurty}
	If $M\in \mathscr{C}(k,n)$, then $|\mathscr{B}(M)|\geq k(n-k)+1$. Furthermore there is a unique (up to isomorphism) connected matroid of size $n$ and rank $k$ for which equality is attained.
\end{teo}

We proceed to a realization of these minimal matroids. They happen to be indeed graphical matroids.

\begin{prop}\label{grafo}
	Let $T_{k,n}$ be the graph given by a cycle of length $k+1$ where one edge is replaced with $n-k$ parallel copies. Then the cycle matroid of $T_{k,n}$ is connected, has cardinality $n$, rank $k$ and exactly $k(n-k)+1$ bases.
\end{prop}

\begin{proof}
	We will use the name \textit{red edges} when we refer to the $n-k$ parallel edges as in the statement. The remaining edges will be called \textit{black edges}. 
	\begin{figure}[H]\label{figura}
		\begin{tikzpicture}  
		[scale=2.5,auto=center,every node/.style={circle,fill=blue!10}] 
		
		\node (a1) at (-0.5,0.86) {};  
		\node (a2) at (-1,0)  {}; 
		\node (a3) at (-0.5,-0.86)  {};  
		\node (a4) at (0.5,-0.86) {};  
		\node (a5) at (1,0)  {};  
		\node (a6) at (0.5,0.86)  {};    
		
		\draw (a1) -- (a2); 
		\draw (a2) -- (a3);  
		\draw (a3) -- (a4);  
		\draw (a4) -- (a5);  
		\draw (a5) -- (a6);  
		\draw (a6) edge[color=red] (a1);
		\draw (a6) edge[bend right=20,color=red] (a1);
		\draw (a6) edge[bend right=-20,color=red] (a1);
		\end{tikzpicture} \caption{$T_{5,8}$}
	\end{figure}
	Observe that we cycle matroid of $T_{k,n}$ does indeed trivially satisfy the cardinality and rank conditions: we have $n$ elements in total and the maximal independent sets are of cardinality $k$. It is also straightforward to verify that this graph is biconnected and hence its cycle matroid is connected. 
	
	Finally, since a basis of the cycle matroid corresponds to a spanning tree on the graph, we notice that we have two kind of spanning trees: those that contain just one red edge, and those that contain none. In the first case, we can choose one among the $n-k$ red edges, and leave out one among of the $k$ black edges. In the second case, no red edges implies that the spanning tree must consist of all black edges. Thus, $(n-k)k+1$ is the total number of spanning trees.
\end{proof}

\begin{obs}
	It is clear from the minimality property of these matroids that the dual of the minimal matroid $T_{k,n}$ is isomorphic to $T_{n-k,n}$. 
\end{obs}

Let us prove now that $T_{k,n}$ is a Lattice Path Matroid \cite{bonin} with a nice structure. Using \cite{knauer2,knauer1} and all the terminology defined within that article, we can see that $T_{k,n}$ coincides with the \textit{snake} $S(n-k,k)$. This result is not used in the sequel but we include it here for the sake of completeness.

\begin{prop}
    The minimal matroid $T_{k,n}$ is isomorphic to the snake matroid $S(n-k,k)$.
\end{prop}

\begin{proof}
    Recall from \cite{knauer2, knauer1} that the snake matroid $S(n-k,k)$ is defined as the lattice path matroid given by upper path $U = (1,n-k+1, n-k+2, \ldots,n-1)$ and lower path $L = (n-k+1, n-k+2,\ldots, n)$.
    Following \cite{knauer2}, one has that the rank of this matroid is exactly $k$ and the cardinality is exactly $n$. It is a connected matroid by \cite[Theorem 3.6]{bonin}. If we prove that the number of bases of this matroid is $k(n-k)+1$, then by the uniqueness established in Theorem \ref{dinoltmurty} the proof would be complete. To this end, notice that the number of bases of such a matroid is equal to the number of paths $P$ that do not go above the path:
        \[ \operatorname{st}(U) = (1, \underbrace{0, \ldots, 0}_{n-k-1}, \underbrace{1, \ldots, 1}_{k-1}, 0).\]
    and do not go below the path:
        \[ \operatorname{st}(L) = (\underbrace{0,\ldots, 0}_{n-k},\underbrace{1,\ldots,1}_{k}).\] 
    \begin{itemize}
        \item If $\operatorname{st}(P)$ has all zeros among the first $n-k$ entries, that means that $P=L$.
        \item If $\operatorname{st}(P)$ has exactly one occurrence of a one among the first $n-k$ entries (it cannot have more occurrences, since otherwise it would be above $U$), we have $n-k$ possibilities of where to put that one. Among the remaining entries there has to be exactly one zero, which gives us $k$ possibilities of where to put that. The total count is then $k(n-k)$ for this case.
    \end{itemize}
    So we have exactly $k(n-k)+1$ bases, as claimed.
\end{proof}

Recall that a \textit{flat} $F$ of a matroid $M$ is a subset such that $\rk(F\cup\{e\})>\rk(F)$ for all $e\notin F$. 

We can use the family $\mathscr{F}(M)$ of all the flats of a matroid $M$ to give an inequality description for the basis polytope of $M$. 

\begin{prop}\label{desig}
	Let $M$ be a matroid on the set $\{1,\ldots,n\}$. Then $\mathscr{P}(M)$ is given by:
	\[\mathscr{P}(M) = \left\{ x\in \mathbb{R}^n_{\geq 0} : \sum_{i=1}^n x_i = \rk(M) \text{ and } \sum_{i\in F} x_i \leq \rk(F)  \text{ for all } F\in\mathscr{F}(M)\right\}. \]
\end{prop}

\begin{proof}
	See for example \cite[Proposition 2.3]{feichtner}.
\end{proof}

In all what follows we will use the name $T_{k,n}$ for the cycle matroid of the graph $T_{k,n}$. This abuse of notation should not cause confusions.

Let us characterize all flats of the matroid $T_{k,n}$. Using the notation of the proof of Proposition \ref{grafo}, we see that there are two types of flats in $T_{k,n}$: those that contain a red edge (and hence all of them), and those that consist of only black edges.

We label all black edges with the numbers $\{1,2,\ldots,k\}$ and the red ones with the numbers $\{k+1,\ldots,n\}$.\\

\begin{itemize}
	\item Those flats that contain all red edges, may contain any number $m\neq k-1$ of black edges. It cannot contain exactly $k-1$, since adding the remaining edge will not increase the rank, thus contradicting the definition of flat. Hence there are $2^k - k$ such flats.
	\item Those flats that do not contain red edges may contain any proper subset of black edges. Hence there are $2^k-1$ such flats.\\
\end{itemize}

Using Proposition \ref{desig} we can formulate now a characterization of $\mathscr{P}(T_{k,n})$ using $2^{k+1}-k-1$ inequalities. However, many of these inequalities are superfluous.

\begin{prop}\label{caract}
	The polytope $\mathscr{P}(T_{k,n})$ is characterized by:
	\[ \mathscr{P}(T_{k,n}) = \left\{ x\in [0,1]^n : \sum_{i=1}^n x_i = k \text{ and }\sum_{i=k+1}^{n} x_i\leq 1 \right\}.\]
\end{prop}

\begin{proof}
	Recall that flats consisting only on black edges are exactly proper subsets $F\subsetneq \{1,\ldots,k\}$, and hence having $\rk(F)=|F|$. According to Proposition \ref{grafo}, any of these flats gives an inequality of the form:
		\[ \sum_{i\in F} x_i \leq |F|.\]
	However these inequalities $2^{k}-1$ inequalities are implied trivially by those of the form:
		\[ x_i \leq 1 \text{ for all } i=1,\ldots,k.\]	
	
	Similarly, the flats containing all red edges are of the form $F'\cup \{k+1,\ldots,n\}$ where $F'$ is a subset of $\{1,\ldots,k\}$ with $|F'|\neq k-1$. Inequalities in this case are of the form:
		\begin{equation}
			\sum_{i=k+1}^n x_i + \sum_{i\in F'} x_i \leq |F'|+1. \label{flatsineq}
		\end{equation} 
	In particular, taking $F'=\varnothing$, one has:
		\[ \sum_{i=k+1}^n x_i \leq 1,\]
	and all of \eqref{flatsineq} are implied by the previous inequalities $x_i\leq 1$ for $i=1,\ldots,k$.
\end{proof}

\begin{obs}
	Since the matroid polytope is an intersection of several halfspaces that are in bijective correspondence with the flats of $M$, one may distinguish those flats whose removal changes the polytope and call them \textit{flacets}, (see for instance \cite{feichtner}). Hence, the preceding result gives a characterization of the flacets of $\mathscr{P}(T_{k,n})$.
\end{obs}

\begin{obs}
	As pointed out before, since $T_{k,n}$ is isomorphic to the snake $S(k,n-k)$, it is possible to find another characterization of the polytope $\mathscr{P}(T_{k,n})$ using the result \cite[Theorem 4.7]{knauer2} stating that matroid polytopes of snakes coincide with the order polytopes of zig-zag posets. In this case, $\mathscr{P}(T_{k,n})$ coincides with the order polytope of the \textit{zig-zag poset} $Z(k,n-k)$ (see \cite{knauer2} for the definition of this poset).
\end{obs}

\begin{obs}
    As we mentioned in the introduction, it is not true that the polytope of every connected matroid of rank $k$ and cardinality $n$ contains a copy of $\mathscr{P}(T_{k,n})$. For example, let $M$ be the cycle matroid of the following graph:
    
    
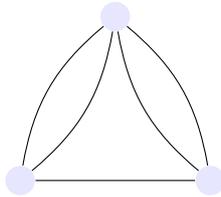
\begin{figure}[H]\label{figura2}
		\begin{tikzpicture}  
		[scale=2.5,auto=center,every node/.style={circle,fill=blue!10}] 
		
		\node (a2) at (-0.5,0.86) {};  
		\node (a3) at (0,1.73)  {}; 
		\node (a1) at (0.5,0.86)  {};    
		
		\draw (a1) -- (a2); 
		\draw (a3) edge[bend right=20,color=black] (a1) ;
		\draw (a3) edge[bend right=-20,color=black] (a1);
		\draw (a3) edge[bend right=20,color=black] (a2);
		\draw (a3) edge[bend right=-20,color=black] (a2);
		\end{tikzpicture} 
	\end{figure}
    This matroid $M$ has $8$ bases, given that the graph has $8$ spanning trees. It is connected, has rank $2$ and cardinality $5$. Also, $T_{2,5}$ has exactly $7$ bases. There is no way we can delete one basis from the set $\mathscr{B}(M)$ and obtain the set of bases of a matroid isomorphic to $T_{2,5}$. At the level of polytopes, this means that no subset of $7$ vertices of $\mathscr{P}(M)$ induces a polytope that is a copy of $\mathscr{P}(T_{k,n})$.
\end{obs}

\section{The Ehrhart Polynomial of $T_{k,n}$}

In this section we give a formula for the Ehrhart polynomial of $\mathscr{P}(T_{k,n})$. Our proofs are elementary and consist in several manipulations of combinatorial identities. In the Appendix we include the proofs of some results that are used throughout our computations. We remark that alternative proofs are possible using the language of generalized hypergeometric functions and hypergeometric transformations \cite{knuth}.\\

We start with our first formula for $D_{k,n}$. An equivalent version of this formula was found in \cite[Theorem 3.8]{knauer2}.

\begin{teo}\label{minehrhart}
	Let $D_{k,n}(t)\in\mathbb{Q}[t]$ be the Ehrhart polynomial of the matroid $T_{k,n}$. Then the following equality holds.
		\[D_{k,n}(t) = \sum_{j=0}^{k-1} \binom{k-1}{j}\binom{n-k-1}{j}\binom{t+n-1-j}{n-1}.\]
\end{teo}

\begin{proof}
	Recall that $D_{k,n}(t)$ is the number of lattice points inside the dilated polytope $t\mathscr{P}(T_{k,n})$. Using Proposition \ref{caract}, this is:
	\begin{align*}
		D_{k,n}(t) &= \#\left(\mathbb{Z}^n \cap t\mathscr{P}(T_{k,n})\right)\\
		&= \# \left\{ x\in [0,t]^n : \sum_{i=1}^n x_i = tk \text{ and }\sum_{i=k+1}^{n} x_i\leq t \right\}.
	\end{align*}
	To count the number of elements of this set, we proceed as follows. Let us fix a number $0\leq j \leq t$ and set the sum $\sum_{i=k+1}^n x_i$ to be exactly $j$. The number of ways to achieve this is exactly the number of ways of putting $j$ indistinguishable balls into $n-k$ distinguishable boxes, which is just $\binom{n-k-1+j}{n-k-1}$.
	
	Now we have to count the number of ways of putting $tk-j$ indistinguishable balls into exactly $k$ distinguishable boxes, each of them having a capacity of $t$. Using Proposition \ref{ballsandboxes} in the appendix one has then:
		\begin{align} 
			D_{k,n}(t) &= \sum_{j=0}^t \binom{n-k-1+j}{n-k-1} \binom{k-1+tk - (tk-j)}{k-1}.\nonumber\\
			&= \sum_{j=0}^t \binom{n-k-1+j}{j} \binom{k-1+j}{j}\label{sumita}
		\end{align}
	Then, by Proposition \ref{doublestick} in the appendix, one gets the result of the statement.
\end{proof}


The formula presented in the preceding Theorem, and the one of equation \eqref{sumita} are useful for computations, but do not show the positivity of the coefficients of $D_{k,n}$. A first step towards that is to notice the following factorization:

\begin{lema}
	The following identity holds:
	\[ D_{k,n}(t) = \binom{t+n-k}{n-k} \sum_{j=0}^{k-1} \frac{n-k}{n-k+j} \binom{t}{j}\binom{k-1}{j}\]
\end{lema}

\begin{proof}
	The proof consists only of sum manipulations starting with equation \eqref{sumita}. Steps on numbered equations are justified below.
		\begin{align}
			D_{k,n}(t) &= \sum_{j=0}^t \binom{n-k-1+j}{j}\binom{k-1+j}{j}\nonumber\\
			&= \sum_{j=0}^t \binom{n-k-1+j}{j}\sum_{i=0}^{k-1} \binom{k-1}{k-1-i}\binom{j}{i}\label{vander}\\
			&= \sum_{i=0}^{k-1} \sum_{j=0}^t \binom{k-1}{i} \binom{n-k-1+j}{j}\binom{j}{j-i}\nonumber\\
			&= \sum_{i=0}^{k-1} \sum_{j=0}^t \binom{k-1}{i} \binom{n-k-1+j}{j-i}\binom{n-k-1+i}{i}\label{ident}\\
			&= \sum_{i=0}^{k-1} \binom{k-1}{i}\binom{n-k-1+i}{i} \sum_{j=0}^t \binom{n-k-1+j}{n-k-1+i}\nonumber\\
			&= \sum_{i=0}^{k-1} \binom{k-1}{i} \binom{n-k-1+i}{i} \binom{t+n-k-1}{n-k+i}\label{hs} \\
			&= \sum_{i=0}^{k-1} \binom{k-1}{i}\frac{n-k}{n-k+i} \binom{t}{j}\binom{t+n-k}{n-k}\label{last}
		\end{align}
	where in \eqref{vander} we used Vandermonde's Identity, in \eqref{ident} the identity $\binom{r}{m}\binom{m}{k}=\binom{r}{k}\binom{r-k}{m-k}$, in \eqref{hs} the Hockey-Stick Identity (also known as the parallel sumation formula \cite{knuth}) and in \eqref{last} just factorials simplifications.
\end{proof}

Observe that from this Lemma we get that $D_{k,n}(t)$ can be written as a product of a polynomial with positive coefficients: $\binom{t+n-k}{n-k}$ and a remaining factor, which we will call $R_{k,n}(t)$. It is:
	\[ R_{k,n}(t) = \sum_{j=0}^{k-1} \frac{n-k}{n-k+j}\binom{t}{j}\binom{k-1}{j} \]

Hence, if we prove that $R_{k,n}$ has positive coefficients, then we will be able to conclude the positivity of the coefficients of $D_{k,n}$. This is done in the following Lemma.

\begin{lema}
	\[ R_{k,n}(t) = \frac{1}{\binom{n-1}{k-1}} \sum_{j=0}^{k-1} \binom{n-k-1+j}{j}\binom{t+j}{j}\]
\end{lema}

\begin{proof}
	We have the following chain of equalities:
	\begin{align}
	\binom{n-1}{k-1} R_{k,n}(t) &= \binom{n-1}{k-1} \sum_{j=0}^{k-1} \frac{n-k}{n-k+j} \binom{t}{j}\binom{k-1}{j}\nonumber\\
	&= \sum_{j=0}^{k-1} \binom{n-1}{k-1} \frac{\binom{n-k-1+j}{j}}{\binom{n-k+j}{j}} \binom{t}{j}\binom{k-1}{j}\nonumber\\
	&= \sum_{j=0}^{k-1} \binom{n-1}{k-1}\binom{k-1}{k-1-j} \frac{\binom{n-k-1+j}{j}}{\binom{n-k+j}{j}} \binom{t}{j}\nonumber\\
	&= \sum_{j=0}^{k-1} \binom{n-1}{k-1-j} \binom{n-k+j}{n-k} \frac{\binom{n-k-1+j}{j}}{\binom{n-k+j}{j}} \binom{t}{j}\label{cruzado}\\
	&= \sum_{j=0}^{k-1} \binom{n-1}{k-1-j}\binom{n-k-1+j}{j}\binom{t}{j}\label{expr}
	\end{align}
	where in \eqref{cruzado} we used the identity $\binom{r}{m}\binom{m}{k} = \binom{r}{k}\binom{r-k}{m-k}$. On the other hand:
	\begin{align}
		\sum_{j=0}^{k-1} \binom{n-k-1+j}{j}\binom{t+j}{j} &= \sum_{j=0}^{k-1}\binom{n-k-1+j}{j}\sum_{i=0}^j \binom{t}{i}\binom{j}{j-i}\label{vanderr}\\
		&=\sum_{i=0}^{k-1}\sum_{j=i}^{k-1}\binom{t}{i}\binom{n-k-1+j}{j}\binom{j}{j-i}\nonumber\\
		&=\sum_{i=0}^{k-1}\sum_{j=i}^{k-1}\binom{t}{i} \binom{n-k-1+j}{j-i} \binom{n-k-1+i}{i}\label{cruzadoo}\\
		&=\sum_{i=0}^{k-1}\binom{n-k-1+i}{i}\binom{t}{i} \sum_{j=i}^{k-1}\binom{n-k-1+j}{j-i}\nonumber\\
		&= \sum_{i=0}^{k-1}\binom{n-k-1+i}{i}\binom{t}{i}\sum_{j=0}^{k-1-i} \binom{n-k-1+i+j}{j}\nonumber\\
		&= \sum_{i=0}^{k-1}\binom{n-k-1+i}{i}\binom{t}{i}\binom{n-1}{n-k+i}\label{hstick}
	\end{align}
	where in \eqref{vanderr} we used Vandermonde's Identity, in \eqref{cruzadoo} we used again $\binom{r}{m}\binom{m}{k} = \binom{r}{k}\binom{r-k}{m-k}$ and in \eqref{hstick} we used the classic Hockey Stick Identity. Observe that \eqref{expr} and \eqref{hstick} are equal, so the result of the statement follows.
\end{proof}

\begin{coro}
	The polynomial $D_{k,n}(t)$ has positive coefficients. If we call $d_{k,n,m} := [t^m] D_{k,n}(t)$, then it holds:
		\begin{align}
			D_{k,n}(t) &= \frac{1}{\binom{n-1}{k-1}} \binom{t+n-k}{n-k} \sum_{j=0}^{k-1}\binom{n-k-1+j}{j}\binom{t+j}{j}\label{formu}\\
			 d_{k,n,m} &= \frac{1}{(n-1)!} \sum_{j=0}^{k-1}\sum_{\ell=0}^j \frac{(k-1)!}{j!} \binom{n-k-1+j}{j} {j+1 \brack{\ell+1}} {n-k+1 \brack{m-\ell+1}}\label{formucoef}
		\end{align} 
\end{coro}

\begin{proof}
	The equation \eqref{formu} is just a consequence of the preceding Lemmas. From this equality, as we said above, the positivity of the coefficients is clear. 
	
	The computation of $d_{k,n,m}$ is a straightforward consequence of the fact that $[t^m]a!\binom{t+a}{a}$ is the Stirling Number of the first kind ${a+1}\brack{j+1}$. 
\end{proof}

\begin{obs}\label{remark}
	Notice that from our formula \eqref{formu} for $D_{k,n}$ it is evident that $D_{k,n}(t-1)$ has nonnegative coefficients.
\end{obs}

In \cite{ferroni} the author proved a combinatorial formula for the coefficients of the Ehrhart polynomial of the hypersimplex $\Delta_{k,n}$, which is the basis polytope of the uniform matroid $U_{k,n}$. Using that result, we can reformulate Conjecture \ref{cota} as follows:

\begin{conj}
	Let $M$ be a connected matroid of rank $k$ and $n$ elements. Let us call $i(M,t)$ its Ehrhart Polynomial. Then for every $m\in\{0,\ldots,n-1\}$ it holds:
	\[ d_{k,n,m} \leq [t^m] i(M,t) \leq e_{k,n,m},\]
	where 
	\begin{align*}
		d_{k,n,m}&=\frac{1}{(n-1)!} \sum_{j=0}^{k-1}\sum_{\ell=0}^j \frac{(k-1)!}{j!} \binom{n-k-1+j}{j} {j+1 \brack{\ell+1}} {n-k+1 \brack{m-\ell+1}},\\
		e_{k,n,m}&=\frac{1}{(n-1)!}\sum_{\ell=0}^{k-1} W(\ell,n,m+1)A(m,k-\ell-1),
	\end{align*}
	denoting $A$ the Eulerian numbers, and $W$ the weighted Lah numbers (defined in \cite{ferroni}).
\end{conj}

\begin{obs}\label{monotonicity}
	This conjecture may make one fall in the temptation of saying that a matroid with more bases has bigger Ehrhart coefficients. This is not true in general. There are two matroids $M_1$ and $M_2$ of rank $3$ and cardinality $7$ such that $M_1$ has $29$ bases and $M_2$ has $30$ bases and yet the normalized volume of $M_1$ is bigger than that of $M_2$. Their set of bases are given by:
		 \begin{align*}
		 \mathscr{B}(M_1) &= \binom{[7]}{3} \smallsetminus \{ \{2,5,7\}, \{3,4,5\}, \{2,4,6\}, \{3,6,7\}, \{1,5,6\},\{1,2,3\}\},\\
		 \mathscr{B}(M_2) &= \binom{[7]}{3}\smallsetminus \{\{1,2,5\},\{1,2,4\},\{1,2,7\},\{1,2,6\},\{1,2,3\} \},
		 \end{align*}
	where $[7]$ denotes the set $\{1,2,3,4,5,6,7\}$ and the binomial coefficient notation stands for all the subsets of cardinality $3$.
	
	Their Ehrhart polynomials are respectively:
	\begin{align*}
		i(M_1,t) &=\frac{1}{6!}\cdot(242t^6 + 1464t^5 + 3860t^4 + 5940t^3 + 5618t^2 + 3036t + 720),\\
		i(M_2,t) &= \frac{1}{6!} \cdot(198t^6 + 1386t^5 + 4050t^4 + 6390t^3 + 5832t^2 + 3024t + 720).
	\end{align*}
\end{obs}

\section{The $h^*$-polynomial of $T_{k,n}$}

As a consequence of Theorem \ref{minehrhart} we have a formula for the $h^*$-polynomial of $T_{k,n}$.

\begin{coro}
	The $h^*$-polynomial of the matroid polytope of $T_{k,n}$ is given by the formula:
		\[h^*(T_{k,n},x) = \sum_{j=0}^{k-1} \binom{k-1}{j}\binom{n-k-1}{j} x^j. \]
	This polynomial is real rooted and thus the $h^*$-vector is log-concave and unimodal.
\end{coro}

\begin{proof}
	 It is a routine computation working with Theorem \ref{minehrhart}, calculating the product $(1-x)^n h^*(T_{k,n},x)$, which by definition is the generating function of the Ehrhart polynomial \cite{beck}.
	 
	 The real-rootedness of this polynomial is a well known fact, see for example the Concluding Remarks in \cite{knauer2}. The log-concavity and the unimodality of the $h^*$-vector are a consequence of this (see for example \cite{branden1}).
\end{proof}

Although the Ehrhart polynomial of $T_{k,n}$ is a bit difficult to work with, the $h^*$-polynomial permits us to obtain some information of the polytope $\mathscr{P}(T_{k,n})$.

\begin{coro}
	The normalized volume of the matroid polytope $\mathscr{P}(T_{k,n})$ is given by:
	\[ \operatorname{Vol}(\mathscr{P}(T_{k,n})) = \binom{n-2}{k-1}.\]
\end{coro}

\begin{proof}
	Since the volume is given by $h^*(T_{n,k},1)$, it suffices to do the computation:
	\begin{align*}
		h^*(T_{k,n},1) &= \sum_{j=0}^{k-1} \binom{k-1}{j}\binom{n-k-1}{j}\\
		&= \sum_{j=0}^{k-1}\binom{k-1}{k-1-j} \binom{n-k-1}{j}\\
		&= \binom{n-2}{k-1},
	\end{align*}
	where in the last step we used Vandermonde's Identity.
\end{proof}

\section{Relaxations of a Matroid}

We will discuss a matroidal operation that behaves nicely with the Ehrhart polynomial of the basis polytope. \\

Recall that if $M$ is a matroid on the ground set $E$ of rank $k$ and cardinality $n$, then a \textit{hyperplane} of $M$ is a coatom in the lattice of flats of $M$. Equivalently, a flat $F\subseteq E$ is said to be a hyperplane if $\rk(F)=k-1$.

If $H\subseteq M$ is a hyperplane and a circuit, then one can \textit{relax} the matroid $M$, declaring that $H$ is a basis. More precisely:

\begin{prop}
	Let $M$ be a matroid with set of bases $\mathscr{B}$ that has a circuit-hyperplane $H$. Let $\widetilde{\mathscr{B}}=\mathscr{B}\cup\{H\}$. Then $\widetilde{\mathscr{B}}$ is the set of bases of a matroid $\widetilde{M}$ on the same ground set as $M$.
\end{prop} 

\begin{proof}
	See \cite[Proposition 1.5.14]{oxley}.
\end{proof}

The operation of declaring a circuit-hyperplane to be a basis is known in the literature by the name of \textit{relaxation}. Many famous matroids arise as a result of this operation on another matroid. For example the \textit{Non-Pappus matroid} is the result of relaxing a circuit-hyperplane on the \textit{Pappus matroid}, and analogously the \textit{Non-Fano} matroid can be obtained by a relaxation of the \textit{Fano} matroid (for some other examples see \cite{oxley}).

Of course, relaxing a circuit-hyperplane doesn't alter the rank of the matroid. It also preserves or increases its degree of connectivity (see \cite[Propositon 8.4.2]{oxley}). 

\begin{lema}
	Let $M$ be a matroid with set of bases $\mathscr{B}$ and a circuit-hyperplane $H$. Let $\widetilde{M}$ be the relaxed matroid. Then, the set of flats $\widetilde{\mathscr{F}}$ of $\widetilde{M}$ is given by:
	\[\widetilde{\mathscr{F}} = \left(\mathscr{F}\smallsetminus\{H\}\right) \cup \{F\subseteq H : |F| = |H|-1\},\]
	where $\mathscr{F}$ is the set of flats of $M$.
\end{lema}

\begin{proof}
	Notice that the rank function $\widetilde{\rk}$ of $\widetilde{M}$ coincides with the rank function $\rk$ of $M$ with the only exception of $\rk(H)+1=\widetilde{\rk}(H)$.
	
	Let $F$ be a flat of $\widetilde{M}$ that is not a flat of $M$. Then $\widetilde{\rk}(F\cup e)>\widetilde{\rk}(F)$ for all $e\notin F$. Since $F\neq H$, we have that $\widetilde{\rk}(F)=\rk(F)$. Notice that there exists an $e$ such that $F\cup e=H$, since otherwise our inequality holds for all $e$ with $\rk$ instead of $\widetilde{\rk}$ and thus contradicting that $F$ is not a flat of $M$. Then $F\subseteq H$ and $|F|=|H|-1$, as claimed.
	
	The reverse inclusion follows from the fact that all those sets are  flats of $\widetilde{M}$.
\end{proof}

This characterization of the flats of the relaxed matroid $\widetilde{M}$ helps us to characterize the matroid polytope by deleting just one inequality in the description of the polytope of $M$. Namely, the precise inequality corresponding to the flat $H$.

\begin{prop}
	Let $M$ be a matroid of rank $k$ and cardinality $n$ with a circuit-hyperplane $H$. Then the matroid polytope of the relaxation $\widetilde{M}$ is given by:
		\[\mathscr{P}(\widetilde{M}) = \left\{ x\in \mathbb{R}^n_{\geq 0} : \sum_{i=1}^n x_i = k \text{ and } \sum_{i\in F} x_i \leq \rk(F)  \text{ for all } F\in\mathscr{F}(M)\smallsetminus \{H\}\right\}. \]
\end{prop}

\begin{proof}
	Using the notation of the preceding Lemma, it suffices to see that the inequalities that come from flats of $\widetilde{M}$ of the form $F=H\smallsetminus h$ with $h\in H$ are superfluous.
	
	Indeed, since in that case $F$ is independent, the inequality $\sum_{i\in F} x_i \leq \rk(F)$ is trivially implied by the inequalities $x_i\leq 1$.
\end{proof}

The following results state the exact relation between minimal matroids and the operation of circuit-hyperplane relaxation on the language of matroid subdivisions.

\begin{teo}
    Let $M$ be a (connected) matroid of rank $k$ and cardinality $n$ with a circuit-hyperplane $H$ and let $\widetilde{M}$ be the relaxed matroid. Then the polytope $\widetilde{\mathscr{P}}$ of $\widetilde{M}$ is obtained by stacking the polytope of the minimal matroid $T_{k,n}$ through a facet of $\mathscr{P}$. 
\end{teo}

\begin{proof}
    Notice that $\widetilde{\mathscr{P}}$ contains all the vertices of $\mathscr{P}$ and an extra vertex corresponding to $H$. If we use the characterization of the polytope of a matroid (see \cite[Theorem 4.1]{GGMS}), we have that $H$ has $k(n-k)$ adjacent vertices, corresponding to the bases of $\widetilde{M}$ (and hence of $M$) that differ in exactly one element with $H$. To prove that there are indeed $k(n-k)$ such bases, let us call $H=\{h_1,\ldots,h_k\}$. Since $H$ is a circuit-hyperplane of $M$, if we call $\{e_1,\ldots,e_{n-k}\}$ the elements in the complement of $H$, we have that
		\[B_{ij} := (H\smallsetminus \{h_i\})\cup \{e_j\}\]
	is a basis of $M$ for each $1\leq i\leq k$ and each $1\leq j\leq n-k$. These correspond to the $k(n-k)$ vertices adjacent to $H$ in $\widetilde{\mathscr{P}}$. Also, for each $i$ and $j$ we have that $B_{ij}$ is adjacent with all $B_{i'j}$ and all $B_{ij'}$ for $i'\neq i$ and $j'\neq j$. All this amounts to say that if we restrict ourselves to the polytope $\mathscr{Q}$ given by the $k(n-k)+1$ vertices given by $H$ and all $B_{ij}$, it is in fact the polytope of a minimal matroid.
\end{proof}

An immediate consequence of the above subdivision is that circuit-hyperplane relaxation behaves nicely with Ehrhart polynomials.

\begin{coro}
	Let $M$ be a connected matroid of rank $k$ and cardinality $n$ with a circuit-hyperplane $H$. Let $i(M,t)$ and $i(\widetilde{M},t)$ denote the respective Ehrhart polynomials of their polytopes. The following equality holds:
		\[ i(\widetilde{M},t) = i(M,t)+D_{k,n}(t-1).\]
	In particular, if $M$ is Ehrhart positive so is $\widetilde{M}$.
\end{coro}

\begin{proof}
	Using the notation of the proof of the preceding Theorem, we know that:
		\[ \widetilde{\mathscr{P}} = \mathscr{P}\cup \mathscr{Q},\]
	and that $\mathscr{P}\cap\mathscr{Q}$ is a facet of $\mathscr{P}$ and $\mathscr{Q}$. So an inclusion-exclusion argument reveals now that:
		\[ i(\widetilde{M},t) = i(M,t) + D_{k,n}(t) - S(t),\]
	where $S(t)$ is the Ehrhart polynomial of the facet of $\mathscr{Q}$ consisting of all the $k(n-k)$ bases of $T_{k,n}$ containing a red edge. It is evident from Proposition \ref{caract} that this facet of $\mathscr{Q}$ can be interpreted as:
		\[\left\{x\in [0,1]^n: \sum_{i=1}^n x_i =k \text{ and} \sum_{i=k+1}^n x_i = 1\right\},\]
	and then the number of integer points in a dilation by the factor $t$ of this facet is given by:
		\[
			S(t) = \#\left\{x\in \mathbb{Z}^n \cap [0,t]^n : \sum_{i=1}^n x_i = kt \text{ and } \sum_{i=k+1}^n x_i = t \right\},
		\]
	from which, using the same \textit{balls and boxes} reasoning, exactly as in the proof of Theorem \ref{minehrhart}, we see that
		\[ S(t) = \binom{n-k-1+t}{n-k-1}\binom{k-1+t}{k-1},\]
	and we have from equation \eqref{sumita} that $D_{k,n}(t)-S(t)$ is equal then to $D_{k,n}(t-1)$. We conclude then the Ehrhart positivity of $\widetilde{M}$ given that $i(M,t)$ is assumed to have positive coefficients, recalling Remark \ref{remark}.
\end{proof}

\begin{obs}
	In this case we have that \textit{adding a vertex} to our matroid polytope does increase Ehrhart coefficients, cf. Remark \ref{monotonicity}. 
\end{obs}

\begin{obs}
    It is worth noting that the case of the presence of a circuit-hyperplane is the only scenario on which one can add just one basis and preserve the matroid structure \cite{mills}. To be precise, if $\mathscr{B}$ is the set of bases of a matroid $M$ and $H$ is a subset such that $\mathscr{B}\sqcup \{H\}$ is also the set of bases of a matroid, this means that $H$ was originally a circuit-hyperplane of $M$. For a proof of this result one can also read \cite[Lemma 6]{truemper}.
\end{obs}

Of course, one has an equivalent version of the above result in the language of $h^*$-polynomials.

\begin{coro}
	If $M$ is a matroid of rank $k$ and cardinality $n$ with a circuit-hyperplane $H$ and $\widetilde{M}$ is the relaxed matroid, then:
		\[ h^*(\widetilde{M},t) = h^*(M,t) + th^*(T_{k,n},t). \]
\end{coro}

\begin{proof}
	The result follows by using the definition of the $h^*$-polynomial as the numerator of the generating function of the Ehrhart polynomial.
\end{proof}

It seems likely that if there is any hope of giving a combinatorial interpretation of the coefficients of the $h^*$-vector of a matroid, then the preceding Corollary might help to build an intuition of what these elements are counting.

\begin{obs}
	Now we can construct examples of non-isomorphic and non-dual connected matroids that have the same Ehrhart Polynomial and the same Tutte polynomial. Notice that if $M$ has a circuit-hyperplane $H$ and $\widetilde{M}$ is the relaxation, then
		\begin{equation}
		 T_{\widetilde{M}}(x,y) = T_M(x,y)-xy+x+y,\label{tutte}
		 \end{equation}
	where $T_M$ and $T_{\widetilde{M}}$ denote the Tutte polynomials (recall that the rank function coincides everywhere except in $H$, so using the definition yields directly to \eqref{tutte}). Hence, picking two non-isomorphic matroids that can be relaxed to the same matroid, one may construct such examples.
	
	For instance, consider the matroids $M_1$ and $M_2$ of rank $3$ and cardinality $7$ whose set of bases consist of $\mathscr{B}_1$ and $\mathscr{B}_2$ given by:
		\[ \mathscr{B}_1 = \binom{[7]}{3}\smallsetminus \left\{ \{1,2,3\}, \{4,5,6\} \right\},\]
		\[ \mathscr{B}_2 = \binom{[7]}{3}\smallsetminus \left\{ \{1,2,3\}, \{3,4,5\} \right\}.\]
	They can be seen to be indeed matroids that are not isomorphic, that can be relaxed twice to obtain the uniform matroid $U_{3,7}$ and hence have the same Ehrhart and the same Tutte polynomial. Yet, their polytopes are not even combinatorially equivalent, since for instance their $f$-vectors are different:
		\[ f_1 = (1, 33, 186, 325, 248, 92, 16, 1),\]
		\[ f_2 = (1, 33, 186, 326, 249, 92, 16, 1). \]
\end{obs}

\section{Appendix}

We have collected in this final section some combinatorial results that are used throughout the proofs of our formulas for $D_{k,n}$.

\begin{prop}\label{ballsandboxes}
	Let $n$ and $c_1\geq \ldots\geq c_k$ be nonnegative integers such that $n\geq \sum_{i=1}^{k-1} c_i$. The number of ways $N$ of putting exactly $n$ indistinguishable balls into $k$ distinguishable boxes of capacities $c_1,c_2,\ldots,c_k$ is given by:
		\[N = \binom{k-1 + \sum_{i=1}^k c_i - n}{k-1}.\]
\end{prop}

\begin{proof}
	Note that instead of thinking of putting balls in a box, we can think of \textit{leaving free space} in a box.
	
	The sum of free spaces in any possible distribution will be exactly $\sum_{i=1}^k c_i - n$. Thus we have to assign free spaces $f_1,\ldots,f_k$ to every box in such a way that their sum is: 
		\[ f_1 + \ldots + f_k = \sum_{i=1}^k c_i - n,\]
	and we are given the constraint $0\leq f_i\leq c_i$, of which the inequalities $f_i\leq c_i$ are superfluous since the constraints $f_i\geq 0$ for all $i$ (it is, \textit{all of them}) already imply that
		\[f_i\leq \sum_{i=1}^k c_i - n \leq c_k \leq c_i.\]
	Hence we just have to count the number of ways to put $\sum_{i=1}^k c_i - n$ indistinguishable balls into $k$ distinguishable boxes, which gives the desired result.
\end{proof}

\begin{lema}[Sur\'anyi's Identity]
	\[ \binom{r+j}{r} \binom{s+j}{s} = \sum_k \binom{r}{k}\binom{s}{k}\binom{j+r+s-k}{r+s}\]
\end{lema}

\begin{proof}
	See \cite[Corollary 2]{szekely}.
\end{proof}

\begin{lema}[Double Hockey-Stick Identity]\label{doublestick}
	\[ \sum_{j=0}^m \binom{r+j}{j}\binom{s+j}{j} = \sum_k \binom{r}{k}\binom{s}{k}\binom{r+s+1+m-k}{r+s+1}\]
\end{lema}

\begin{proof}
	We proceed using Sur\'anyi's Identity:
		\begin{align*}
			\sum_{j=0}^m \binom{r+j}{j}\binom{s+j}{j} &= \sum_{j=0}^m \binom{r+j}{r}\binom{s+j}{s}\\
			&= \sum_{j=0}^m \sum_{k=0}^s \binom{r}{k}\binom{s}{k}\binom{j+r+s-k}{r+s}\\
			&=\sum_{k=0}^s \binom{r}{k}\binom{s}{k} \sum_{j=0}^m \binom{j+r+s-k}{r+s}\\
			&=\sum_{k=0}^m \binom{r}{k}\binom{s}{k} \binom{r+s+1+m-k}{r+s+1},
		\end{align*}
	where in the last step we used the classic Hockey-Stick identity.
\end{proof}

\section*{Acknowledgements}

The author wants to thank his Ph.D supervisor, Prof. Luca Moci, and the reviewers for the careful reading and helpful comments to improve many aspects of this article, and to Kolja Knauer for several useful comments regarding snake matroids.

\bibliographystyle{plain}
\bibliography{bibliopaper}

\end{document}